\documentclass[a4paper]{amsart}
\usepackage{xr-hyper}
\usepackage{nameref}
\usepackage[usenames]{color}
\usepackage[ps2pdf]{hyperref}
\usepackage{url}
\newtheorem{dummy}{realdumb}[section]
\newtheorem{theorem}[dummy]{Theorem}

\newtheorem{proposition}[dummy]{Proposition}
\theoremstyle{definition}               

\newtheorem{question}[dummy]{Question}

\newtheorem{remark}[dummy]{Remark}
{\leqno{\rm ({\thedummy})} $$} \numberwithin{equation}{dummy}
\theoremstyle{plain}

\begin{document}
\title{On manifolds with nonhomogeneous factors}

\author{M. C\'ardenas, F.F. Lasheras, A. Quintero}
\address{Departamento de Geometr\'{\i}a y Topolog\'{\i}a, Universidad de Sevilla,
Apdo 1160, 41080-Sevilla, Spain}
\email{mcard@us.es, lasheras@us.es, quintero@us.es}
\author{D. Repov{\v{s}}}
\address{Faculty of Mathematics and Physics, 
and Faculty of Education, University of Ljubljana,
P.O. Box 2964, Ljubljana 1001, Slovenia}
\email{dusan.repovs@guest.arnes.si}



\begin{abstract}
{We present simple examples of finite-dimensional connected homogeneous spaces (they are actually topological manifolds) with nonhomogeneous and nonrigid factors. In particular, we give an elementary solution of an old problem in general topology concerning homogeneous spaces.}
\end{abstract}

\date{\today}

\maketitle
\section{Introduction}\label{sect1}
A topological space $X$ is said to be {\it homogeneous} if for every pair of points $x,y\in X$ there exists a homeomorphism 
$h:(X,\{x\})\to (X,\{y\})$. This very classical topological
 notion became very important when in the 1960's Bing and Borsuk proved that
in dimensions below $3$, homogeneity can actually detect topological manifolds among all finite-dimensional absolute neighborhood retracts (ANR's). 

Bing and Borsuk also conjectured  that
this is
true in all dimensions, and this conjecture
remains a formidable open problem (in dimension $3$ it  implies the Poincar\' e Conjecture). 
Recently, homogeneity has gained renewed attention among geometric topologists,
 since it turned out that the so-called Busemann $G$-spaces (which have also been conjectured to be topological manifolds)
possess homogeneity among other key properties \cite{Bryant1, halversonrepovs}. 

It was  in our recent investigations of Bing-Borsuk and Busemann Conjectures
 that we came upon some observations
on the homogeneity
%
%
of products of nonhomogeneous spaces which we have collected in the present paper. 
In particular, 
we give (uncountably many) connected nonrigid finite-dimensional positive answers to the following question from \S 1.7 listed on
page 125 of \cite{topproblist}: "Is there a nonhomogeneous (compact) space whose square is homogeneous?". 
Clearly,
there is a dimensional restriction to such examples, namely 
$n\geq 3$. 

Several positive answers to this
question are already known. In 2003 a nonconnected example was given by Rosicki \cite{rosicki} in the realm of topological groups. Earlier, an infinite-dimensional connected rigid example was constructed by van Mill \cite{vanmill} in 1981, whereas in 1983 Ancel and Singh \cite{Ancel-Singh} constructed finite-dimensional rigid examples $X$ with $dim\ X\geq 4$ and Ancel, Duvall and Singh \cite{Ancel-Duvall-Singh} produced such an example also for the case $dim\ X=3$. Recall that a space is said to be \emph{rigid} if it does not have any self-homeomorphism other than the identity.

The results in this paper provide alternative finite-dimensional answers to the question above which are easier to construct than the rigid ones, as a straightforward application of the theory of decompositions of manifolds.  

In a more general setting, we say that a space $X$ is \emph{$k$-homogeneous}, $k\geq 2$, if for any given $k$-element sets $\{x_1,\dots , x_k\}$ and $\{y_1,\dots , y_k\}$, there exists a homeomorphism $h:X\to X$ such that $h(x_i)=y_i$ for $i=1,\dots, n$. The case $k=2$ is simply referred to as \emph{bihomogeneity}. Our results can also be used to get examples of manifold factors that fail to be
bihomogeneous and $(k\ge 3)$-homogeneous 
since any of them implies homogeneity.

\section{Wild cells of arbitrary dimensions and codimensions}\label{sect2}
In geometric topology, a \emph{wild $k$-cell}  in $\mathbb{R}^n$ ($0 < k <n$) is a topological embedding of the 
unit $k$-ball of $\mathbb{R}^k$ which cannot be mapped onto the 
canonical embedding $B^k \subset \mathbb{R}^k \subset \mathbb{R}^n$ 
by a homeomorphism of $\mathbb{R}^n$ onto itself. 
First examples of  wild $1$-cells in $\mathbb{R}^3$ (called  \emph{wild arcs}) were constructed by 
Artin and Fox (\cite{artinfox}).
In fact, there are uncountably many wild arcs
(see \cite{foxharrold,lomonaco,myers}). 
In \cite{myers} different arcs were distinguished by the fundamental groups of their complements: i.e., two arcs 
$\alpha$, $\beta$ are not equivalent  if and only if 
\break
$\pi_1(S^3 \setminus \alpha)\not\cong \pi_1(S^3 \setminus \beta)$. Recall that two arcs are called \emph{equivalent} if there is a self-homeomorphism of $\mathbb R^3$ taking one arc to the other.
Notice that we can consider these arcs to be wild also
in $S^3$.

Well-known methods based on elementary properties of the suspension of a space lead to the construction 
of wild cells in arbitrary dimensions.
To illustrate this
we shall 
give some details. Let $\mathcal{F}_{3,1}$ be any uncountable family of wild arcs in $S^3$ 
such that their complements in $S^3$ are not simply connected
(for instance,
the one given in \cite{myers}), and let $\alpha\in\mathcal {F}_{3,1}$. 
Then for each $k \geq 1$ one can
construct from $\alpha$ a sequence of wild arcs $(\alpha_k) \subset S^{3+k}$.
Indeed, if we  already have a wild arc $\alpha_{k-1}$  in $S^{3+(k-1)}$, then 
by Corollary 2.6.4 of \cite{davermanvenema},
the sphere $S^{3+k}$ is homeomorphic 
to the suspension of the quotient space $S^{3+(k-1)}/\alpha_{k-1}$. 

Now, using Lemma 2.7.2 of \cite{davermanvenema}, 
let $\alpha_k$ be the arc in $S^{3+k}$ that corresponds to the suspension of the class of the points of 
 $\alpha_{k-1}$ in the quotient  $S^{3+(k-1)}/\alpha_{k-1}$. Notice that $S^{3+k} \setminus \alpha_k$ is homotopically equivalent to $S^{3+(k-1)} \setminus \alpha_{k-1}$, 
 hence $\alpha_{k}$
 is wild.
So for each $n\geq 3$, there is a family $\mathcal F_{n,1}$ of uncountably many distinct wild arcs in $S^n$.

Now, given $n\geq 3$
and
$0<k<n$, let $\alpha$ be a wild arc in $S^{n-k+1}$ from the collection $\mathcal{F}_{n-k+1,1}$.
By Lemma 1.4.1 of \cite{davermanvenema}, 
the $(k-1)$-th suspension $\Sigma^{k-1}\alpha$ of $\alpha$ is a wild $k$-cell in $S^n$. 
Hence, for each $n\geq 3$ and $0<k<n$, there is a family $\mathcal F_{n,k}$ of uncountably many wild $k$-cells embedded in $S^n$.

Notice that the $k$-cells in $\mathcal F_{n,k}$ are cell-like non-cellular sets, for $n\geq 3$ and $0<k<n$ (the failure of cellularity follows from Exercise 2.7.4 of
 \cite{davermanvenema}
 and Exercise 2.6.2.(a) of \cite{rushing}). As above, these $k$-cells can be taken to be embedded either in $\mathbb R^n$ or in $S^n$.

\section{Products of Generalized Manifolds}\label{sect3}
A generalized $n$-manifold $X$ is defined as a finite-dimensional Euclidean neighborhood retract (ENR) whose local $\mathbb Z$-homology groups agree with those of the Euclidean $n$-space, i.e. 

 $$H_*(X,X \setminus \{x\};\mathbb Z)\cong H_*(\mathbb R^n;\mathbb R^n \setminus \{0\};\mathbb Z) \ \ 
 \hbox{\rm for all} \ \ x\in X.$$ 

By Theorem 6 of \cite{raymond} and \cite{bredonwildermflds},
 if $X^{k}\times Y^{l}$ is a homology $n$-manifold then $X^{k}$ and $Y^{l}$ are homology $k$- and $l$-manifolds, respectively,
where $k+l=n$ 
(see also Theorem 2 of \cite{brahana}). 

Also, by Problem K.3 on p. 30 of \cite{Hu}, $X\times Y$ is a metrizable ANR if and only if $X$ and $Y$ are metrizable ANR's. Combining both results we get the following:
\begin{proposition}\label{3}
 If $X^{k}\times Y^{l}$ is a generalized $(k+l)$-manifold then $X^k$ and $Y^l$ are generalized $k$- and $l$-manifolds, respectively. 
\end{proposition}

\section{Examples}\label{sect4}
Unless otherwise stated, all manifolds in this section will be assumed to be connected.
\begin{theorem}\label{1}
 There exist uncountably many  distinct
 topological
$n$-manifolds $M^n$ such that $M^n=X\times Y$, where $X$ is a
nonhomogeneous nonmanifold factor,
if and only if $n\geq 4$.
\end{theorem}

\begin{proof}
\noindent
($\Leftarrow$)
If $n\geq 4$, then by Theorem 1 of \cite{andrewscurtis} it suffices
 to pick any wild arc $\alpha\subset\mathbb R^{n-1}$ such that $\mathbb R^{n-1} \setminus \alpha$ is not simply connected, 
from the family $\mathcal F_{n,1}$ defined in Section \ref{sect2}, and consider the quotient space
 $X^{n-1}=\mathbb R^{n-1}/\alpha$. 
Then $X^{n-1}$ is a generalized $(n-1)$-manifold with one singular point (hence a nonhomogeneous space), since
the space $X^{n-1}$ fails to be locally Euclidean at $\pi(\alpha)\in X$, 
where $\pi:\mathbb R^{n-1}\to X^{n-1}=\mathbb R^{n-1}/\alpha$ is the quotient map. 
On the other hand, $X\times\mathbb R$ is homeomorphic to $\mathbb R^n$, so $X^{n-1}$ is an $n$-manifold factor. 
By letting $\alpha$ range over the uncountable family $\mathcal F_{n,1}$, we 
obtain uncountably many distinct examples. Moreover, if one wants to obtain a closed manifold $M$, one just observes that 
$(S^{n-1}/\alpha)\times \mathbb{R}$ is an $n$-manifold, and so $(S^{n-1}/\alpha)\times S^1$ is a closed $n$-manifold. 

In a similar way, but in a more general setting, one can apply Theorem 1.1 of \cite{bryanteuclmodcell}: Let $D$ be any $k$-cell in the family $\mathcal F_{n,k}$ defined in Section \ref{sect2}, with $n\geq 4$ and $0<k<n$. Then $(\mathbb R^{n-1}/D)\times\mathbb R$ is homeomorphic to $\mathbb R^n$. Again, the space $X=\mathbb R^{n-1}/D$ is not a manifold since it has a unique singular point which makes the space nonhomogeneous. 
By varying $D$ in $\mathcal{F}_{n,k}$ we get uncountably many distinct
examples.
\par
\noindent
($\Rightarrow$)
Let $n\leq 3$ and $M^n=X^k\times Y^l$ so that $k+l=n$. Then by Proposition~\ref{3}, $X^k$ and $Y^l$ are generalized $k$- and $l$-manifolds, respectively.
 Since $M$ is connected, we may assume that $0<k$ and
 $l<3$. Hence these low-dimensional generalized manifolds $X$ and $Y$ are actually genuine manifolds
 (see \cite{repovs92} or Ch. IX of \cite{wilderbook})
 and are therefore also homogeneous.
\end{proof}

\begin{remark}
 For $n\geq 5$, more examples can be constructed. Given a topological $(n-2)$-manifold 
$M^{n-2}$ choose any cell-like usc-decomposition $\mathcal G$ such that 
 $X^{n-2}=M^{n-2}/\mathcal G$ is finite-dimensional. (This dimensionality
 condition  is necessary due to  examples of Dranishnikov \cite{dranishnikov} and Dydak-Walsh \cite{dydakwalsh}). 
By Theorem 26.8 of \cite{davermanbook}, $X^{n-2}\times\mathbb R^2$ has the Disjoint Disks Property and hence, by Edwards' Theorem (see e.g. Theorem 2.2 of \cite{halversonrepovs}):
$$
N^n=X^{n-2}\times \mathbb R^2
$$
\noindent
is a topological $n$-manifold, whereas $X$ is nonhomogeneous if one assumes that the singular set 
is not dense in X, i.e. $\overline{S(X)}\neq X$. Here $S(X)$ denotes
the {\it singular set} of $X$, i.e., the set of all points 
in $X$ having no Euclidean neighborhood.
\end{remark}

\begin{theorem}\label{2}
 There exist uncountably many topological $2n$-dimensional manifolds $M^{2n}$ such that $M^{2n}=X\times X$, where
 $X$ is a nonhomogeneous manifold factor, if and only if $n\geq 3$.
\end{theorem}
\begin{proof}
\noindent
($\Leftarrow$)
For $n\geq 3$, we apply Corollary 3 from \cite{bass} (see also \cite{smith}) 
to a cell-like decomposition $\mathcal G$ of a manifold $M$ of dimension $dim\ M\geq 3$ 
in order to obtain that $(M/\mathcal G)\times (M/\mathcal G)$ is homeomorphic to $M\times M$. 
Hence, it is enough to take $M=S^n$ and $\mathcal G$ to be the cell-like decomposition, whose 
only nondegenerate element is one of the $k$-cells from
the family $\mathcal F_{n,k}$ which was
defined in Section \ref{sect2}.
\par
\noindent
($\Rightarrow$)
Let $n\leq 2$. Given $N^{2n}=X\times X$, then (as above) it follows by Proposition~\ref{3} that $X$ is a generalized $n$-manifold. 
Therefore, for $n=2$, $X$ is a generalized $2$-manifold and hence a surface; while for $n=1$, $X$ is a 
generalized $1$-manifold and hence a circle.
Recall that $(n<3)$-dimensional homology manifolds are genuine  manifolds (see Ch. IX of \cite{wilderbook}).
\end{proof}

\begin{remark}
In fact, Bass' result used in the proof of Theorem \ref{2} shows that given two manifolds of dimensions 
$\geq 3$ and cell-like decompositions $\mathcal G$ and $\mathcal G'$ of $M$ and $N$,
 respectively, satisfying certain mild conditions, 
 it follows that  $M\times N \cong (M/\mathcal G)\times (N/\mathcal G')$. Hence, these constructions provide 
affirmative answers in all dimensions $\geq 6$ to the second question from \S 1.7 posed on p. 125 of
 \cite{topproblist}: ``Can the product of two nonhomogeneous spaces be homogeneous?''.
\end{remark}

\section{Epilogue}\label{sect5}
\begin{question}
What can one say about homogeneous continua with nonhomogeneous factors in arbitrary dimensions? More explicitly, we state the following questions:
\begin{enumerate}
\item Can an $(n\leq 3)$-dimensional homogeneous continuum $K$ be written as a product  $K=X\times Y$, 
where at least one of the factors $X$ and $Y$ is not  homogeneous?
\item Can an $(n\leq 5)$-dimensional homogeneous continuum $K$ be written as a product of two nonhomogeneous factors?
\item Can an $(n\leq 5)$-dimensional homogeneous continuum $K$ be written as $K=X\times X$, where $X$
is not homogeneous?
\end{enumerate}
\end{question}

\begin{question}
Does the Logarithmic Law hold for homogeneous compact ANR's, i.e. does the following equality hold:
$$
\rm{dim} (X\times Y)=\rm{dim} X +\rm{dim} Y?
$$
According to the proof sketched in \cite{fedorchuk}, the so-called {\it Pontryagin surfaces} $T_p$ are homogeneous. 
Recall that these celebrated
compacta, which have the property that
$\dim T_p$= 2 for all prime $p$, but 
$\dim(T_p \times T_q) = 3$, whenever $p\neq q$, show that the Logarithmic Law fails if $X$ and $Y$ are not ANR's. 
Recent work \cite{Bryant1} was believed to lead to a positive answer to Question~5.2 (see \cite{fedorchuk}). 
However, last year Bryant discovered a serious gap in the proof of Theorem 2 from \cite{Bryant1}.
\end{question}

\begin{remark}
The most famous problem
still open in decomposition theory is the classical 
R. L. Moore Problem
from
the 1930's, concerning the characterization of topological $n$-manifolds. It
asks if
every (finite-dimensional) cell-like decomposition
$\mathbb{R}^n/\mathcal G$ of $\mathbb{R}^n$ is a topological factor of $\mathbb{R}^{n+1}$, i.e. 
$$(\mathbb{R}^n/\mathcal{G})\times\mathbb{R}\cong \mathbb{R}^{n+1}$$
(see \cite{halverson-repovs} for a recent survey on this difficult
problem).

 In connection with the Moore Problem
 we mention Problem 9.5 of \cite{quinnproblemlist}, which asks
 if the product of a homology manifold and $\mathbb R$ is always homogeneous? 
Many examples (in particular, those in  Theorems~\ref{1} and \ref{2}) give partial affirmative answers to both
of these
 questions, but there are still far more examples to be considered.
\end{remark}

\section*{Acknowledgments}
The first three authors were supported by the  MTM project 2010-20445.
The fourth author was supported by the SRA grants P1-0292-0101 and J1-2057-0101.
We gratefully acknowledge remarks
and suggestions from John L. Bryant, Janusz Prajs and both referees.


\newcommand{\noopsort}[1]{} \def\cprime{$'$} \newcommand{\printfirst}[2]{#1}
  \newcommand{\singleletter}[1]{#1} \newcommand{\switchargs}[2]{#2#1}
\providecommand{\bysame}{\leavevmode\hbox to3em{\hrulefill}\thinspace}
\providecommand{\MR}{\relax\ifhmode\unskip\space\fi MR }
\providecommand{\MRhref}[2]{%
  \href{http://www.ams.org/mathscinet-getitem?mr=#1}{#2}
}
\providecommand{\href}[2]{#2}

\end{document}